\documentclass[a4paper,twoside,english]{amsart}
\usepackage[T1]{fontenc}
\usepackage[latin9]{inputenc}
\usepackage{babel}
\usepackage{verbatim}
\usepackage{amsthm}
\usepackage{amsmath}
\usepackage{amssymb}
\usepackage{placeins} %Per controllare meglio la posizione di immagini e tabelle
\usepackage[unicode=true,
 bookmarks=true,bookmarksnumbered=true,bookmarksopen=false,
 breaklinks=false,pdfborder={0 0 1},backref=false,colorlinks=false]
 {hyperref}
\hypersetup{
 pdfauthor={U. Zannier}}

\makeatletter

\pdfpageheight\paperheight 
\pdfpagewidth\paperwidth

\theoremstyle{plain}
\newtheorem{thm}{\protect\theoremname}[section]
 \newcommand\thmsname{\protect\theoremname}
 \newcommand\nm@thmtype{theorem}
 \theoremstyle{plain}

  \theoremstyle{remark}
  \newtheorem{rem}[thm]{\protect\remarkname}
  \theoremstyle{definition}
  \newtheorem*{example*}{\protect\examplename}
  \theoremstyle{definition}
  
  \theoremstyle{plain}
  
  \theoremstyle{plain}
  \newtheorem{prop}[thm]{\protect\propositionname}
  \theoremstyle{plain}

\makeatletter
  \theoremstyle{definition}
  \newtheorem{my@rem}[thm]{Remark}
  \renewenvironment{rem}{\begin{my@rem}}{\end{my@rem}}
\makeatother

\makeatother

\newcommand{\abs}[1]{\left| #1 \right|}
\newcommand{\ceiling}[1]{\left\lceil #1 \right\rceil}
  \providecommand{\examplename}{Example}
  \providecommand{\lemmaname}{Lemma}
  \providecommand{\propositionname}{Proposition}
  \providecommand{\remarkname}{Remark}
  \providecommand{\theoremname}{Theorem}
\providecommand{\theoremname}{Theorem}
 \providecommand{\corollaryname}{Corollary}
\usepackage[left=3.1cm,right=3.5cm,top=3cm,bottom=3cm]{geometry}

\usepackage{mathtools}

\def\N{{\Bbb N}}

\def\R{{\Bbb R}}

\def\Z{{\Bbb Z}}

\def\Aa{{\mathcal A}}
\def\H{{\mathcal H}}

\def\C{{\Bbb C}}

\def\d{{\rm d}}

\def\CVD{{\hfill\hfil{\lower 2pt\hbox{\vrule\vbox to 7pt
{\hrule width  5pt\varphifill\hrule}\varphirule}}}\par}

\begin{document} 

\title{A note on the zeroes of the Fredholm Series}
%% Some Unlikely Intersections  in Ell. Bill. ??? 
\author{Umberto  Zannier}
\thanks{Umberto Zannier, Scuola Normale Superiore, Piazza dei Cavalieri 7, 56126 Pisa, ITALY\\ umberto.zannier@sns.it}
%\date{December, 2016}
\maketitle
\vspace{-0.8cm}\begin{center}With an appendix by F. Veneziano\end{center}\vspace{1cm}

\centerline{\it Dedicated to the memory of Edoardo Vesentini}

\bigskip

%\date{\today}

\bigskip

{\sc Abstract}.  { The issue had been raised whether the so-called  {\it Fredholm series} $z+z^2+\dotsb+z^{2^n}+\dotsb$ has infinitely many zeroes in the unit disk. We provide an affirmative answer, proving that in fact every complex number occurs as a value infinitely many times. }

\medskip

\section{Introduction} 

%%% remark massed + mahler + difficulty with rouche etc.  + difficulty of exponential sums + method suitable for other values of f(z)    remarks massed distribution + trascendenza + in dip. algebraic ?

Let $D=\{z\in \C:|z|<1\}$ denote the complex unit disk and let 
\begin{equation}\label{E.F}
f(z)=z+z^2+z^4+\dotsb+ z^{2^n}+\dotsb , \qquad z\in D.
\end{equation}

This function  is proposed in several books  as a simple instance of a holomorphic function on $D$ which cannot be holomorphically extended to any connected open domain of $\C$ containing $D$ strictly  (see e.g. \cite{T}, p. 159, or \cite{WW}, p. 98).  
It is often called the {\it Fredholm series}, though apparently Fredholm had considered a different series;  the terminology was introduced  by A. J. Kempner in the belief that Fredholm studied it, and then this was followed by several authors: see  J. Shallit's paper \cite{S}, especially pp.~161--162, for an accurate description. (Fredholm is also mentioned in 
%as it was exhibited in 1890 by I. Fredholm as an example of a holomorphic series having the unit circle as a natural boundary (see 
R. Remmert's book  \cite{Re}, p. 254  but concerning indeed a different series sharing the said property.)   % though I do not have any reference for this terminalogy.  
%Indeed several books propose it as a simple instance of a holomorphic function on $D$ which cannot be holomorphically extended to any connected open domain of $\C$ containing $D$ strictly  (see e.g. \cite{T}, p. 159, or \cite{WW}, p. 98).  
We shall not change here this convention,  though perhaps the function should be called {\it Kempner series}.

It satisfies the functional equation $f(z)=z+f(z^2)$. 

The function became better known also through the work of K. Mahler \cite{Mah}, who considered its values at (nonzero) algebraic points in $D$ and proved their transcendency (together with  similar results for a much wider class of functions). In fact, concerning the function itself, the transcendency of the values at points $1/n$, $n\in\N^+$,  had been already established by Kempner \cite{K}; these values are interesting and related to iterated paperfolding (see \cite{S}, pp.  162--163).
All of this gave rise to many-sided generalisations and several works, both of arithmetical and functional nature, by a number of authors. See  %for instance Mahler's paper \cite{Mah}, and see 
also D. Masser's book \cite{M} for an account of Mahler's method in particular, and for some  references,  and see \cite{Mah2} p. 132 for Mahler's reminiscences on this topic. 

\medskip

In the book \cite{M} the issue is raised (see Ex. 3.15) {\it whether the function has infinitely many zeroes in $D$}, and it is remarked that Mahler found 16 zeroes (apparently using a pocket calculator). Actually already Mahler enquired about the location of the zeroes in the paper \cite{Mah3}, commenting that there are probably zeroes ``{\it in every neighborhood of the unit circle}''.  It is easy to locate the  (unique)  real nontrivial zero $-0.658626...$ on looking at signs, but otherwise the location of zeroes seems not obvious to detect, especially due to the fairly complicated oscillatory nature of $f(z)$ near the boundary of $D$ (which is related to exponential sums with widely growing arguments - see \cite{Mah3} and see  Remarks \ref{R.R}, \ref{R.R2} below). Of course in any disk $rD$  of fixed radius $r<1$ there are only finitely many zeroes of $f$, so Mahler's expectation literally amounts to the existence of an infinity of zeroes.

Note that several `general' methods for studying the zeroes (and the values)  of holomorphic functions (e.g. Picard Theorems, Nevanlinna Theory)  usually apply to functions with singularities, or  defined on the whole complex plane,  and do not seem to be immediately helpful  here since $f(z)$ is regular on the whole $D$ and has the unit circle as a natural boundary. (See  however W. Hayman's book \cite{Hay}, especially Chs. 5 and  6, and also W. Rudin's book \cite{R}, Chs. 15, 17, for some results about the distribution of zeroes of functions holomorphic in $D$. See Remark  \ref{R.R} below for some comments on this.)  

\medskip

It is the main purpose of this note to prove that indeed there are infinitely many zeroes of $f(z)$ in $D$, confirming the expectation of Mahler. Actually we shall prove a sharpening of this result,  namely that the function attains every complex number as a value infinitely often, and even arbitrarily near $1$.

\medskip

By the results of Mahler any of the zeroes, apart from $0$,  is a transcendental number. Masser asked if for instance one can prove sharper results, e.g. of algebraic independence, or if one can study more accurately  the distribution of zeroes. We have no answer to the former issue; the latter maybe admits some kind of answer through the present method, but we haven't carried out any analysis in this direction (see the final remarks for a few comments).

\medskip

This  note contains also an  appendix by Francesco Veneziano, who kindly followed some  requests of mine, on performing  several computations concerning this problem; in the Appendix he lists some approximate zeroes of the function and locates a zero of a related function which appears in the proof.  Also, he represented  in Figure 1 below some zeroes of a truncated series, rescaled with respect to the Poincar\'e distance in $D$, from the origin.

\bigskip

Before stating our results, we introduce some notation. We denote as usual $e(w)=\exp(2\pi iw)$ and we let $\H=\{w\in \C:\Im w>0\}$ denote  the upper-half plane, with which we shall work in place of $D$. Indeed,  for $w\in \H$ we set 

 \begin{equation}\label{E.Fw}
F(w):=f(e(w))=e(w)+e(2w)+e(4w)+\ldots , \qquad  w\in \H.
\end{equation}

This function is holomorphic in $\H$ and satisfies the functional equations
\begin{equation}\label{E.fe}
F(w+1)=F(w),\qquad F(w)=e(w)+F(2w).
\end{equation}

Consider now the function
\begin{equation}\label{E.G}
G(w):=\sum_{l=1}^{\infty} e\left(-\frac{1}{2^l}\right)\left(e\left(\frac{w}{2^l}\right)-1\right),\qquad w\in\H.
\end{equation}
 Note that for $w$ in a given compact set $K\subset\C$ we have    $|e(w/2^l)-1|\ll_K 2^{-l}$, hence  the series defining $G$  converges uniformly and boundedly on compact sets of $\C$, so in fact $G$ is in particular holomorphic on the whole $\C$.
 
 We further define
\begin{equation}\label{E.S}
S(w)=F(w)+G(w),\qquad w\in \H.
\end{equation}
 
 We have
 
 \begin{prop} \label{P.P}  Let $V=S(\H)\subset \C$ be the image  of $S$ on $\H$. Then, for every open disk $U$ centered at $1$ and for every $v\in V$,  the Fredholm function $f(z)$ assumes on $D\cap U$  infinitely many times the  value $v$. 
 \end{prop}
 
% \medskip
 
 The following result also shows that the set $V$ of Proposition \ref{P.P} is stable by certain translations, thus is in a sense `large'.
 
  \begin{prop} \label{P.P2}  Let $V=S(\H)\subset \C$ be as above  the image  of $S$ on $\H$. Then  $V+\Z=V$ and moreover for each $m\in\Z$ we have $V+c_m\subset V$, where $c_m$ is as in \eqref{E.cm} below.
 \end{prop}
 
 A study of the numbers $c_m$ carried out in \S \ref{SS.cm} will in fact imply the following sharpening of the proposition (which essentially then becomes a lemma): 
 
 \begin{thm} \label{T.V} We have $V=\C$, that is, the function $S(w)$ attains every complex value.
 \end{thm}
 
% \medskip
 
  In particular, combining these results we obtain the following consequence,  which contains as a corollary the main motivation for this note:
 
 \begin{thm}\label{T.T}
 For every open disk $U$ centered at $1$, the Fredholm series attains on $D\cap U$ every complex number as a value  infinitely many times.
 \end{thm}
 
%The   arguments below automatically  yield that $f(z)$ in fact attains all  values in any neighbourhood of $1$. 
So  in particular $f(z)$  has infinitely many zeroes in $D$ and even in $D\cap U$. It may be that the    methods  allow to replace $U$ with any disk centered on  the unit circle but we have not verified this.

\bigskip

{\bf Acknowledgements}.  I thank David Masser for raising several issues and for remarks and references; similarly I thank Pietro Corvaja. I thank Francesco Veneziano for some computer calculations, especially  helpful when writing the paper, and for providing the Appendix. I further thank Jeffrey  Shallit for pointing out to me a correct historical account of the terminology, and some references.

%\medskip

\section{Some  formulae}  
We now let $\theta\in\R$ be real and $n>0$ be a positive integer, and seek an approximation formula for $F(w2^{-n}+\theta)$.  On iterating \eqref{E.fe} we have

\begin{equation}\label{E.app}
\begin{split}
F\left({w\over 2^n}+\theta\right) & = e\left({w\over 2^n}+\theta\right)+e\left({w\over 2^{n-1}}+2\theta\right)+\dotsb +e\left({w\over 2}+2^{n-1}\theta\right)+F(w+2^n\theta)
\\
& =e\left({w\over 2^n}\right)e(\theta)+e\left({w\over 2^{n-1}}\right)e(2\theta)+\dotsb +e\left({w\over 2}\right)e(2^{n-1}\theta)+F(w+2^n\theta).
\end{split}
\end{equation}

We shall use this formula, with suitable choices for the involved quantities, to approximate $F$  with convenient  functions.

\medskip

\subsubsection{A Ramanujan sum and related ones} As remarked by Mahler in \cite{Mah3} the behaviour of $F(w)$ (for rational $\Re  w$) is linked  with certain Gaussian periods. Here we adopt a related but different viewpoint, constructing vanishing {\it Ramanujan sums}.

We let $q=3^k$, where $k\ge 2$ is an integer.  Note  the Euler function value $\phi(q)=2\cdot 3^{k-1}$.

Since $2$ is a primitive root modulo $3$ \footnote{In what follows this choice is immaterial and there is wide possibility of changing it.} and since $2^{2}\not\equiv 1\pmod{9}$, it follows easily (e.g. by induction on $k$)  that $2$ is a primitive root modulo $q$, namely the powers $2^m$, $0\le m\le \phi(q)-1$ form a complete set of residues modulo $q$ and coprime to $q$.  Since $q$ is not squarefree, it follows  (see e.g. \cite{HW}, \S 16.6, especially Thm. 271) that, for any integer $s$ coprime with $q$,  the {\it Ramanujan sum},  i.e.  of the  primitive $q$-th roots of unity, vanishes, which amounts to:
\begin{equation}\label{E.ram}
\sum_{m=0}^{\phi(q)-1}e\left({2^ms\over q}\right)=0,  \qquad \gcd(s,q)=1,
\end{equation}
the same equation actually holding when the summation for $m$  runs through a complete set of residues modulo $\phi(q)$.

\medskip

We now choose  an integer $a\ge 1$ and   $k=2^a$, a power of $2$. We put 
\begin{equation*}
\theta_0={1\over q},\qquad n_0=\phi(q)=2\cdot 3^{2^a-1}.
\end{equation*}

%and we want to analyse the sums $S_r(n_0,\theta_0)$.  From \eqref{E.ram} we first get 
%\begin{equation}\label{E.S0}S_0(n_0,\theta_0)=0.\end{equation}

We also note that $q\equiv 1\pmod {2^{a+2}}$, as follows easily by induction on $a\ge 1$. Hence for $l\le a+2$ the ratio $(1-q)/2^l$ is an integer. Also, we have  for any integer $n\ge l$,
 \begin{equation*}
2^{n-l}\theta_0\equiv2^n{(1-q)\over 2^l}\theta_0\pmod \Z, \qquad 0\le l\le a.
\end{equation*} 
In particular, let $n=n_0$, so $2^{n}\equiv 1\pmod q$. Then  for $l\le a\le n_0$, on writing $2^{n_0}=1+q\ \cdot\ $integer,  this congruence yields 
 \begin{equation}\label{E.modZ}
2^{n_0-l}\theta_0\equiv {(1-q)\over 2^l q}={\theta_0\over 2^l}-{1\over 2^l}\pmod \Z, \qquad 0\le l\le a.
\end{equation}

\medskip

 \subsubsection{Preparing an  approximate formula}  Setting  $s=1$ into \eqref{E.ram} we obtain
  \begin{equation*}
 \sum_{l=0}^{n_0-1}e(2^l\theta_0)=0,
 \end{equation*}
 whence, setting   $\theta=\theta_0$, $n=n_0$ into \eqref{E.app} we derive,  taking into account  that  $2^{n_0}\theta_0\equiv\theta_0\pmod\Z$ and that $F$ is $\Z$-periodic,

 \begin{equation*}
 F\left({w\over 2^{n_0}}+\theta_0\right) =\sum_{l=0}^{n_0-1} e(2^l\theta_0)\left(e\left({w\over 2^{n_0-l}}\right)-1\right)+F(w+\theta_0).
 \end{equation*}
 
 Further, on  writing $n_0-l$ in place of $l$ in the sum and using \eqref{E.modZ}  for $1\le l\le a$, this becomes
 
  \begin{equation}\label{E.app2}
 F\left({w\over 2^{n_0}}+\theta_0\right) =\sum_{l=1}^{a} e\left({\theta_0 -1\over 2^l}\right)\left(e\left({w\over 2^{l}}\right)-1\right)+\sum_{l=a+1}^{n_0} e(2^{n_0-l}\theta_0)\left(e\left({w\over 2^{l}}\right)-1\right)+F(w+\theta_0).
 \end{equation}
 
 \medskip
 
\section{Proofs of main assertions}
 
 \begin{proof}[Proof of Proposition \ref{P.P}]
 
Let   $K\subset \H$ be a given compact subset of $\H$, to which we shall refer by a subscript to indicate dependence of the implicit constants which shall appear. For $w\in K$ we have
\begin{equation*}
|F(w+\theta_0)-F(w)|\ll_K {1\over q},\quad \left|e\left({w\over 2^{l}}\right)-1\right|\ll_K{1\over 2^l},\quad  \left| e\left({\theta_0 -1\over 2^l}\right)-e\left(-{1\over 2^l}\right)\right|\ll {1\over q}\quad (l\ge 1),
\end{equation*}
whence, by \eqref{E.app2},  again for $w\in K$, 
\begin{equation}\label{E.app3}
\left|F\left({w\over 2^{n_0}}+\theta_0\right) -S(w)\right|\ll_K\sum_{l=a+1}^\infty\left|e\left({w\over 2^{l}}\right)-1\right|+{1\over q}\ll_K{1\over 2^a}.
\end{equation}

The proof is now easily completed by a quite standard argument. Let  $v\in V$ be any  value of $S$ on $\H$, say $v=S(w_0)$ for some $w_0\in\H$.  Let $C$ be a circle inside $\H$ and centred at  $w_0$ such that $S(w)\neq v$ for $w\in C$, and choose  $K$ to be a closed disk inside $\H$,  centred at $w_0$,  and of radius larger than that of $C$.  

 The function $S(w)-v$ is holomorphic in $K$ and never zero on $C$,  hence its absolute value attains on $C$ a minimum $\delta>0$. 
 
Put now, for any integer $a\ge 1$,  $F_a(w)=F(2^{-n_0}w+\theta_0)$, where $n_0,\theta_0$ are as above, so by \eqref{E.app3}, $|F_a(w)-S(w)|\le c_K2^{-a}$ for $w\in K$, where $c_K$ is a constant, i.e. independent of $a$. Now  suppose that $a$ is so large that  $2c_K<\delta 2^a$. Then on $C$ we have 
\begin{equation*}
|(F_a(w)-v)-(S(w)-v)|<{\delta\over 2}< \min_{w\in C}|S(w)-v|.
\end{equation*}

By Rouch\'e's theorem (see e.g. \cite{T}) the functions $S(w)-v$, $F_a(w)-v$ have the same number of zeroes inside $C$, so since the former vanishes at $w_0$ (which is the centre of  $C$) the latter has at least a zero $\zeta_a$  inside $C$. 

This means that $F(2^{-n_0}\zeta_a+\theta_0)=v$. Now, the points $2^{-n_0}\zeta_a+\theta_0\in\H$, for varying $a$, %have imaginary parts
 tend to $0$ (because $\zeta_a\in K$), so the corresponding images through $e(w)$ tend to $1$ (but are different from $1$ as they lie in $D$). % in particular  they constitute an infinite set. 
 This proves that, if $U\subset \C$ is any neighbourhood of $1$,  $f(z)$ attains the value $v$ at infinitely many points of $D\cap U$, concluding the argument. \end{proof}

\medskip

\subsection{About the function \texorpdfstring{$S(w)$}{S(w)}}  In this subsection we develop some properties of the function $S(w)$, in particular of its image on $\H$, above denoted $V$. 

It will be equivalent and  notationally convenient to work with the function $S_1(w):=S(w+1)$. By definition we have
\begin{equation}\label{E.S1}
S_1(w)=F(w+1)+G(w+1)=F(w)+\sum_{l=1}^\infty\left(e\left({w\over 2^l}\right)-e\left(-{1\over 2^l}\right)\right).
\end{equation}

 For integers $m\in\Z$,  also define constants $c_m$ by
 \begin{equation}\label{E.cm}
 c_m:=\sum_{l=1}^\infty\left(e\left({m\over 2^l}\right)-1\right).
 \end{equation}
 Hence, setting $H(w):=\sum_{l\ge 1}\left(e\left(w/2^l\right)-1\right)$,  we may rewrite \eqref{E.S1} as 
 \begin{equation}\label{E.S1H}
S_1(w)=F(w)+\sum_{l=1}^\infty\left(e\left({w\over 2^l}\right)-1\right)+\sum_{l=1}^\infty\left(1-e\left(-{1\over 2^l}\right)\right)=F(w)+H(w)-c_{-1}.
\end{equation}
Note that $H$ is holomorphic in the whole $\C$ and $c_m=H(m)$.

As Masser pointed out to me, the function $F(w)+H(w)$ appears already in S. Ramanujan's work (see \cite{H}, p. 39); it also appears (as a function on $D$) in  Exercise 3.11, p. 44 of Masser's book \cite{M} and in the previous article of lectures by  Masser in \cite{M2} (see Lecture 3).  It is observed that it satisfies a functional equation when (in the present notation) $w$ is changed into $2w$. Indeed, we easily find that
\begin{equation}\label{E.S1fe}
S_1(2w)=S_1(w)-1,
\end{equation}
 This already shows that  $V+\Z=V$. Note  also that this yields (Masser's observation)  that the function $S_1(w)+(\log w/\log 2)$ (which is well-defined in $\H$ on agreeing for instance on the value $\log i=\pi/2$) is holomorphic on $\H$ and invariant under  $w\to 2w$, as pointed out in the above quoted sources.
 
 Let us now seek  other functional equations.
 
 Setting $w+m2^s$ in place of $w$, for integers $m\in\Z$, $s\ge 0$, we obtain
 \begin{equation}\label{E.S1fe2}
 \begin{split}
S_1(w+m2^s) &=S_1(w)+\sum_{l=s+1}^\infty e\left({w\over 2^l}\right)\left(e\left({m\over 2^{l-s}}\right)-1\right)\\
& = S_1(w)+ \sum_{l=1}^\infty \left(e\left({w\over 2^{l+s}}\right) -1\right)  \left(e\left({m\over 2^l}\right)-1\right) +c_m\\
& = S_1(w)+\Delta_{m,s}(w)+c_m,
\end{split}
\end{equation}
say.  Note that for $w$ in a compact subset $K$ of $\H$, we have 
\begin{equation}\label{E.D}
\Delta_{m,s}(w):=\sum_{l=1}^\infty \left(e\left({w\over 2^{l+s}}\right) -1\right)  \left(e\left({m\over 2^l}\right)-1\right) =O_K\left({1\over 2^s}\right).
\end{equation}

  \begin{proof}[Proof of Proposition \ref{P.P2}]   We have already noted above that, in view of \eqref{E.S1fe}, the image of $S$ is invariant under translation by integers, which yields the first claim of the proposition.
  
  Let now $K$ be any compact disk inside $\H$ and consider  formula  \eqref{E.S1fe2}.    Note that by \eqref{E.D}, for $s\to\infty$ (and fixed $m$ for instance)  the functions $\Delta_{m,s}$ converge uniformly and boundedly to $0$ for $w\in K$. Then, exactly by the same argument as in the proof of Proposition \ref{P.P},  if $v=S_1(w_0)$ is a value of $S_1$ at a point $w_0$ in the interior of $K$, then for large $s$  (and  fixed $m$) the function $S_1(w+m2^s)-c_m$ assumes as well that value.   Hence $S_1$ (and therefore $S$) assumes the value $v+c_m$, proving the contention. 
  \end{proof} 
  
  \medskip 
  
  Before going ahead, we pause to give a proof of the special case of Theorem \ref{T.T} in which $\C$ is replaced  by $\R$ as a set of possible values. This proof is a bit simpler than the general case and does not require the considerations of \S \ref{SS.cm}.

  \begin{proof}[Proof of Theorem \ref{T.T} for real values]   By Proposition \ref{P.P} it suffices to prove that $\R\subset  V$ (i.e. that $S_1$  attains all real numbers as values  on $\H$). Also, by Proposition \ref{P.P2}  applied with $m=-1$, it suffices to show that $\R-c_{-1}\in V$;  namely, using \eqref{E.S1H}, it suffices to show that  $F(w)+H(w)$ assumes any real value in $\H$.   Let us look at the values for purely imaginary $w=it$, $t\in\R$.  These values are real and by the formula $S_1(2w)=S_1(w)-1$ this set of real values is invariant by any integer translation. Thus it suffices to prove that it contains an interval of length $1$; on the other hand this is clear  by the same formula  $S_1(2w)=S_1(w)-1$, since the imaginary line is connected. 
   \end{proof}

  \subsection{About the constants \texorpdfstring{$c_m$}{cm}} \label{SS.cm} 
  Note that $c_{2m}=c_m$ and that     $\overline{c_m}=c_{-m}$.   %%% continuity of conjugation ? irrationality ? 
  As remarked above, $c_m$ is just the value of $H(z)$ at $z=m$. These numbers seem interesting on their own and we wonder whether they have irrationality of transcendency properties. We cannot prove any neat assertion of this type, but some small information in this direction will appear below.

    Proposition \ref{P.P2}  says that $V$ is stable under translation by the semigroup generated by $\Z$ and the $c_m$. Let $\sigma_m:=c_m+c_{-m}$ noting it  is real, and  in particular we find that $V$ is stable under  translation by the real semigroup generated by $\Z$ and the $\sigma_m$.  We shall study a bit  this semigroup, denoted here $\Gamma$.
    
    \medskip
    
    Now, it turns out that a related sequence has better approximation properties, useful for our purposes. 
    
    Define $k_l:=e(1/2^l)-1$. We have  $k_0=0$, $k_1=-2$, $k_2=i-1$, $k_3=(\sqrt 2-1+i\sqrt 2)/2$ and $k_l=k_{l+1}(k_{l+1}+2)$. Setting $\mu_l:=|k_l|^2$ we also have $\mu_l=2-2\cos(\pi/2^{l-1})$  and $\mu_{l+1}=2-\sqrt{4-\mu_l}$.

    Further,  $e(m/2^l)=(k_l+1)^m$ and $c_m=\sum_{l=1}^\infty ((k_l+1)^m-1)$. For $m\in\Z$ we now set
    \begin{equation}\label{E.c(m)}
    c(m)=\sum_{l=1}^\infty k_l^m.
    \end{equation}
    
    Note that the $k_l$ tend to $0$ exponentially so this formula yields  good approximations of the $c(m)$   (by  linear recurrence sequences) on taking  truncations of the series.
    
    These numbers are linear combinations of the $c_m$ with integer coefficients and conversely. In fact 
    $c(m)=\sum_{l=1}^\infty (e(m/2^l)-{m\choose 1}e((m-1)/2^l)+...)=c_m-{m\choose 1}c_{m-1}+\ldots +(-1)^{m-1}c_1$. Similarly, $c_m=\sum_{l=1}^\infty (k_l^m+{m\choose 1}k_l^{m-1}+\ldots +{m\choose 1}k_l)=c(m)+{m\choose 1}c(m-1)+\ldots +{m\choose 1}c(1)$.

    \medskip
    
    For the sake of curiosity, we also note the recurrence $c(m)=\sum_{l=0}^\infty k_{l+1}^m(k_{l+1}+2)^m=c(2m)+2{m\choose 1}c(2m-1)+4{m\choose 2}c(2m-2)+\ldots + 2^{m-1}{m\choose 1}c(m+1)+2^mc(m)$, whence  
     \begin{equation*}
    (1-2^m)c(m)=\sum_{h=1}^m2^{m-h}{m\choose h}c(m+h),
    \end{equation*}
 so the  $c(m)$ satisfy several linear relations with rational coefficients. These correspond in fact to the more evident relations $c_{2m}=c_m$. We wonder whether there are other such linear relations holding among them.
    
    \medskip
    
    Now, for our purposes it is more useful to consider the numbers 
  \begin{equation*}
  \sigma(m):=\sum_{l=1}^\infty(k_l+\bar k_l)^m=\sum_{l=1}^\infty \left(e\left({1\over 2^l}\right)+e\left(-{1\over 2^l}\right)-2\right)^m.
  \end{equation*}

  Observe that, for $2h\neq m$, $k_l^h\bar k_l^{m-h}+k_l^{m-h}\bar k_l^h$ equals an integer linear combination of terms $e(s/2^l)+e(-s/2^l)-2=\sigma_s$: this is seen on expanding $k_l^h=(e(1/2^l)-1)^h$ with the binomial theorem (and similarly for the complex conjugate). The same holds for a term $k_l^h\bar k_l^h$ in case $2h=m$. On summing over $l$ we deduce that $\sigma(m)$ is an integer linear combination of $\sigma_1,\ldots ,\sigma_m$.    (One can also write down an explicit formula in terms of Chebyshev polynomials.) 
    
% Recall now the numbers $\sigma_m=c_m+c_{-m}=\sum_{l=1}^\infty ((k_l+1)^m+(\bar k_l+1)^m-2)$.
  
    Now, $(-1)^m\sigma(m)$ is of the shape $4^m+2^m+O((2-\sqrt 2)^m)$.  We conclude that  for growing $m$, either one among $\sigma_1,\ldots,\sigma_m$ is irrational, or the least common denominator of these numbers grows to infinity.  In fact, if all $\sigma_j$ would be rationals with a bounded denominator, the said relation would imply the same for $\sigma(m)$, whence, since $0<2-\sqrt 2<1$,  for large $m$ we would have  $\sigma(m)=(k_1+\bar k_1)^m+(k_2+\bar k_2)^m$, which is plainly false. 

We in turn deduce that the semigroup $\Gamma$ generated by $\Z$ and the $\sigma_m$ is dense in $\R$. In fact, if some $\sigma_m$ is irrational this follows from well-known (easy)  theorems, and otherwise it follows from the deduction just obtained.

\subsection{Concluding arguments}  We can now use the results just obtained  to complete the proof of Theorems \ref{T.V} and  \ref{T.T}.

\begin{proof}[Proof of Theorem \ref{T.V} and Theorem  \ref{T.T}]  In view of Proposition \ref{P.P}, it suffices to prove Theorem \ref{T.V}. Let  $z_0\in V$ and let $A$ be an open disk with centre $z_0$, entirely contained in $V$ (which exists since $S(w)$ is holomorphic). %, and denote as above by $\Gamma\subset \R$ the semigroup  generated by $\Z$ and the $\sigma_m$, so 
Recall that $V$ is stable under translation by the semigroup $\Gamma\subset \R$, so  in particular $A+\Gamma\subset V$.  
Since $\Gamma$ is dense in $\R$ we deduce that $V$ contains a whole horizontal strip around the line $\Im w=\Im z_0$. %of the shape $u<\Im w<u+\delta$ for some $u\in\C$ and some $\delta>0$.
 In particular, $V$ contains the line $\Im w=\Im z_0$ itself.  Hence $V$ contains any line $\Im w=t_0 $ as soon as it contains a point in such a line. But since $V$ is invariant under translation by $c_{\pm 1}\neq 0$, thus by any multiple  $sc_{\pm 1}$ (any integer $s>0$, any sign),  the imaginary parts of elements of $V$ are unbounded both from above and from below. But the set of such imaginary parts is connected, hence is the whole $\R$. Combining this with the above deduction, we conclude that $V=\C$, as required.% (and sufficient)  for both results.
\end{proof}

%On writing $(e({1\over 2^l})+e(-{1\over 2^l})-2)^m=\sum_{h=0}^m{m\choose h}(-1)^{m-h}2^{m-h}(e({1\over 2^l})+e(-{1\over 2^l}))^h$ and observing that  $(e({1\over 2^l})+e(-{1\over 2^l}))^h= $ 

 \bigskip

 \begin{rem} \label{R.R}  Concerning the distribution of zeroes of $f(z)$, one could  also think  for instance of using criteria such as Theorem 15.23 of \cite{R}, which would predict a convergent sum $\sum(1-|\alpha_i|)$ for the zeros $\alpha_i$, {\it provided}  $\int_{0}^1\log^+|f(re(\theta))|\d\theta$ remained  bounded as $r\to 1^-$.  
 
 However one may show that  this last condition does not hold  for the present function, that is, $f$ does not belong to the class denoted $N$ (Nevanlinna) in \cite{R}. We give a very brief sketch for this claim.  One first may expand  $f(re(\theta))$ for $r=e^{-c/2^n}$  ($1/2<c\le 1$), using the first $n$ terms and bounding the remainder by a constant (see e.g. \eqref{E.app} above).  Approximating  in turn $r^m=1+O(m/2^n)$,  %$r^m=1-(cm/2^n)+(cm/2^n)^2/2-...$, 
 $m=1,2,\ldots ,2^{n-1}$,  for each of these $n$ terms, one is reduced to estimate the exponential sums $s(\theta)=s_n(\theta)=e(\theta)+e(2\theta)+\ldots +e(2^{n-1}\theta)$ in absolute value.  (See also \cite{Mah3} for such approximate formulae.) 
 
 In conclusion,  if the said integral would be bounded then it is easy to see that,  in particular,  the measure of the set  %these sums  would be bounded, say $<B$,  outside a set
  $\Aa=\Aa_n=\{\theta\in [0,1]: |s(\theta)|>\sqrt n/2\}$ would tend to $0$, hence would be
  $<\epsilon$, for any  prescribed $\epsilon >0$  and large enough $n$. %(and $B=B(\epsilon)$ independent of $n$). 
  Now, however, we easily find  the moments $\int_{0}^1|s(\theta)|^2\d\theta=n$, $\int_0^1|s(\theta)|^4\d\theta=2n^2-n$.  But on the one hand $\int_{\Aa}|s|^2\ge n-(\sqrt n/2)^2=3n/4$, on the other hand $(\int_{\Aa}|s|^2)^2\le \epsilon \int_{\Aa}|s^4|\le 2\epsilon n^2$. This is false for  $\epsilon<9/32$, proving the claim. 
    
  %Despite this, 
  We do not know if nevertheless  the series $\sum (1-|\alpha_i|)$ converges (though we are inclined to believe it does not).  %Maybe it is more sensible to seek functions of the zeros which involve the Poincar\'e metric on $D$ (so in practice expressed as a simple function of  $\log (1-|\alpha_i|)^{-1}$).
  \end{rem} 
 
 \begin{rem} \label{R.R2} Quite possibly there are other methods to prove  the results of the paper (and maybe more). For instance, one could apply Ahlfors' theory (see especially \cite{Hay}, Ch. 5). Also, on approximating $f(z)$ as indicated in the previous remark (or see \cite{Mah3}), one should obtain that for $r=e^{-c/4^n}$,  ($1/2<c\le 1$),  $f(re(\theta))$ is approximated by $s_{2n}(\theta)$ up to a bounded function. In turn, we have $s_{2n}(\theta)=s_n(\theta)+s_n(2^n\theta)$, and on changing $\theta$ to $\theta_a:=\theta+(a/2^n)$ for an arbitrary integer $a$, we get $s_{2n}(\theta_a)=s_n(\theta_a)+s_n(2^n\theta)$. On varying $a$, the first term essentially varies along the whole set of values of $s_n$, whereas the second term remains constant. So one  nearly reduces to study $s_n(u)+s_n(v)$ for {\it independent } variables $u,v$, and one can also iterate (and compare with a compatible probabilistic distribution).  This possibly leads to a better understanding of the image of $f(re(\theta))$ for fixed $r$ approaching $1$, and in turn to an understanding of the distribution of the values attained by the function.
 \end{rem}
 
\appendix
\setcounter{secnumdepth}{0}
\section{Appendix by Francesco Veneziano}

\subsection{Finding a zero of \texorpdfstring{$S(w)$}{S(w)}}
Here we will prove computationally that the holomorphic function $S(w)$ defined in~\eqref{E.S} and expressed by the series
\begin{equation*}\label{S.serie}
 S(w)=\sum_{n=0}^{\infty}\left( e(w 2^n)+e\left(\frac{w-1}{2^{n+1}}\right)-e\left(-\frac{1}{2^{n+1}}\right)\right)
\end{equation*}
has a zero in the upper-half plane. This fact, together with Proposition~\ref{P.P}, provides an alternative proof of the main assertion that the Fredholm series has infinitely many zeroes, independent from the study of the quantities $c_m$ carried out in Section~\ref{SS.cm}.

%The strategy will be of proving explicit bounds on the derivatives of $S(w)$ and finding numerically a point $w_0$ and a radius $\rho$ such that $\abs{S(w_0)}$ is smaller than the absolute value of $S$ on the circle with centre $w_0$ and radius $\rho$. An application of Rouché's theorem used already in Section~\ref{} shows then that $S$ must have a zero within that circle.

%Of course we already know from the fact that $S$ is holomorphic on the upper-half plane and has the real line as a natural boundary that bounds of the shape $\abs{\frac{S^{(k)}(w_0)}{k!}}\ll \left(\frac{1}{\Im w_0}+\epsilon\right)^k$ must hold, but in order for the numerical argument to work we need explicit inequalities.

Thanks to uniform convergence on compacts, we can write
\begin{equation*}\label{S.derivate}
 S^{(k)}(w)=\sum_{n=0}^{\infty}\left((2^{n+1}\pi i)^k  e(w 2^n)+\left(\frac{\pi i}{2^{n}}\right)^k e\left(\frac{w-1}{2^{n+1}}\right)\right).
\end{equation*}
Writing $t=\Im w$ we have
\begin{equation}\label{S.derivate.abs}
 \abs{S^{(k)}(w)}\leq(2\pi)^k\sum_{n=0}^{\infty}2^{k n} e^{-2^{n+1}\pi t}+\pi^k\sum_{n=0}^{\infty}2^{-k n} e^{-\frac{\pi t}{2^n}}.
\end{equation}

%By bounding from below the exponential $2^x$ with its second order Taylor polynomial we see that, for every $n\geq a\geq 0$ we have\begin{equation*} 2^n\geq 2^a\left(1-a\log2+\frac{(a\log2)^2}{2}+n\log2+\frac{(\log2)^2}{2}n(n-a)\right).\end{equation*}

Set now $a=\log_2\frac{k}{2\pi t}$, and let $n$ be an integer, $n\geq \ceiling{a}+1$.
%\begin{align*} -2^{n+1}\pi t&\leq -k\left(1-a\log2+\frac{(a\log2)^2}{2}\right)-kn\log2-\frac{(\log2)^2}{2}kn(n-a)\\ &\leq -\frac{k}{2}-kn\log2-\frac{(\log2)^2}{2}kn, \end{align*}and thus
By  standard calculus we find
 \begin{equation}\label{stima1}
  2^{k n}e^{-2^{n+1}\pi t}=e^{kn\log2-2^{n+1}\pi t}\leq e^{-\frac{k}{2}}e^{-\frac{(\log2)^2}{2}kn}.
 \end{equation}
Splitting the first sum of \eqref{S.derivate.abs} at the summand $\ceiling{a}$ we get
\begin{equation*}
 \abs{S^{(k)}(w)}\leq(2\pi)^k\sum_{n=0}^{\ceiling{a}}2^{k n} e^{-2^{n+1}\pi t}+(2\pi)^k\sum_{n=\ceiling{a}+1}^{\infty}2^{k n} e^{-2^{n+1}\pi t}+\pi^k\sum_{n=0}^{\infty}2^{-k n} e^{-\frac{\pi t}{2^n}}.
 \end{equation*}
 Estimating the first and third terms with geometrical sums, putting these terms together, and  using \eqref{stima1} to estimate the middle term  we get, for $k\ge 2$, 
% \begin{equation*} \leq (2\pi)^k\frac{2^{k(\ceiling{a}+1)}-1}{2^k-1}+\left(\frac{2\pi}{\sqrt{e}}\right)^k\frac{e^{-\frac{(\log2)^2}{2}k(\ceiling{a}+1)}}{1-e^{-\frac{(\log2)^2}{2}k}}+\frac{\pi^k}{1-2^{-k}}.  \end{equation*}  Putting the first and third summands together and assuming now $k\geq 2$ we have
  \begin{equation*}
 \abs{S^{(k)}(w)}  \leq \frac{(2\pi)^k}{2^k-1}2^{k(\ceiling{a}+1)}+3^{k+1}\leq \frac{4}{3}(\pi 2^{a+2})^k+3^{k+1}.
\end{equation*}
From the definition of $a$ and the elementary inequality $k!\leq \frac{k^k}{e^{k-1}}$ %follows that
%\begin{equation*} 2^{ak}=\frac{k^k}{(2\pi t)^k}\leq \frac{k! e^{k-1}}{(2\pi t)^k},\end{equation*}
we eventually  obtain
\begin{equation}\label{stima2}
 \abs{S^{(k)}(w)}\leq\left(\frac{2e}{t}\right)^k k! +3^{k+1}.
\end{equation}
Consider now a point $w_0$ on the upper-half plane, with imaginary part $t_0$. For a point $w$ at distance $\rho$ from $w_0$ using \eqref{stima2} we have that
\begin{align*}
 \abs{S(w)}&=\abs{\sum_{k=0}^\infty \frac{S^{(k)}(w_0)}{k!}(w-w_0)^k}\geq \abs{S'(w_0)}\rho-\abs{S(w_0)}-\sum_{k=2}^{\infty}\frac{\abs{S^{(k)}(w_0)}}{k!}\rho^k\\
 &\geq \abs{S'(w_0)}\rho-\abs{S(w_0)}-\sum_{k=2}^{\infty}\left(\frac{2e\rho}{t_0}\right)^k-3\sum_{k=2}^{\infty}\frac{(3\rho)^k}{k!}\\
 &=\abs{S'(w_0)}\rho-\abs{S(w_0)}-\frac{4e^2\rho^2}{t_0(t_0-2e\rho)}-3e^{3\rho}+9\rho+3.
\end{align*}
Now we can finally set $w_0=-0.177323882+0.144626388i$ and $\rho=0.002$ and check numerically that $\abs{S(w)}>\abs{S(w_0)}$ on the circle of centre $w_0$ and radius $\rho$, thus proving that within that disk lies a zero of $S(w)$.

\subsection{Some zeroes of the Fredholm series}
As shown, the Fredholm series has infinitely many zeroes. We list here the approximate values of a few of them, other than the zero at $z=0$ and the only other real zero at $z=-0.65862675\dotso$
\begin{footnotesize}
\begin{align*}
 -0.93753706&\pm0.33482855 i, & -0.92536435&\pm0.35488765 i, & -0.82600557&\pm0.55115255 i,\\
-0.68520628&\pm0.67053411 i, & -0.55421602&\pm0.82701155 i, & -0.36921064&\pm0.92165325 i,\\
0.12031484&\pm0.93460594 i, & 0.14635991&\pm0.98490932 i, & 0.18459117&\pm0.95835110 i,\\
0.39186275&\pm0.89825762 i, & 0.74745993&\pm0.65768729 i, & 0.77543377&\pm0.61813413 i.
\end{align*}
\end{footnotesize}
Of course the zeroes tend to accumulate towards the unit circle. Figure~\ref{zeri} shows the zeroes of the 13th partial sum which lie in the interior of the unit circle (there are 1126 of them); in order to visualize them better, we have rescaled their distances from the origin according to the hyperbolic metric in the Poincar{\'e} disk model.
In other words, each zero $\rho e^{i\theta}$ in the unit circle has been mapped in the picture to the point $\log\left(\frac{1+\rho}{1-\rho}\right) e^{i\theta}$.

\begin{figure}[h!]
 \includegraphics[scale=0.9]{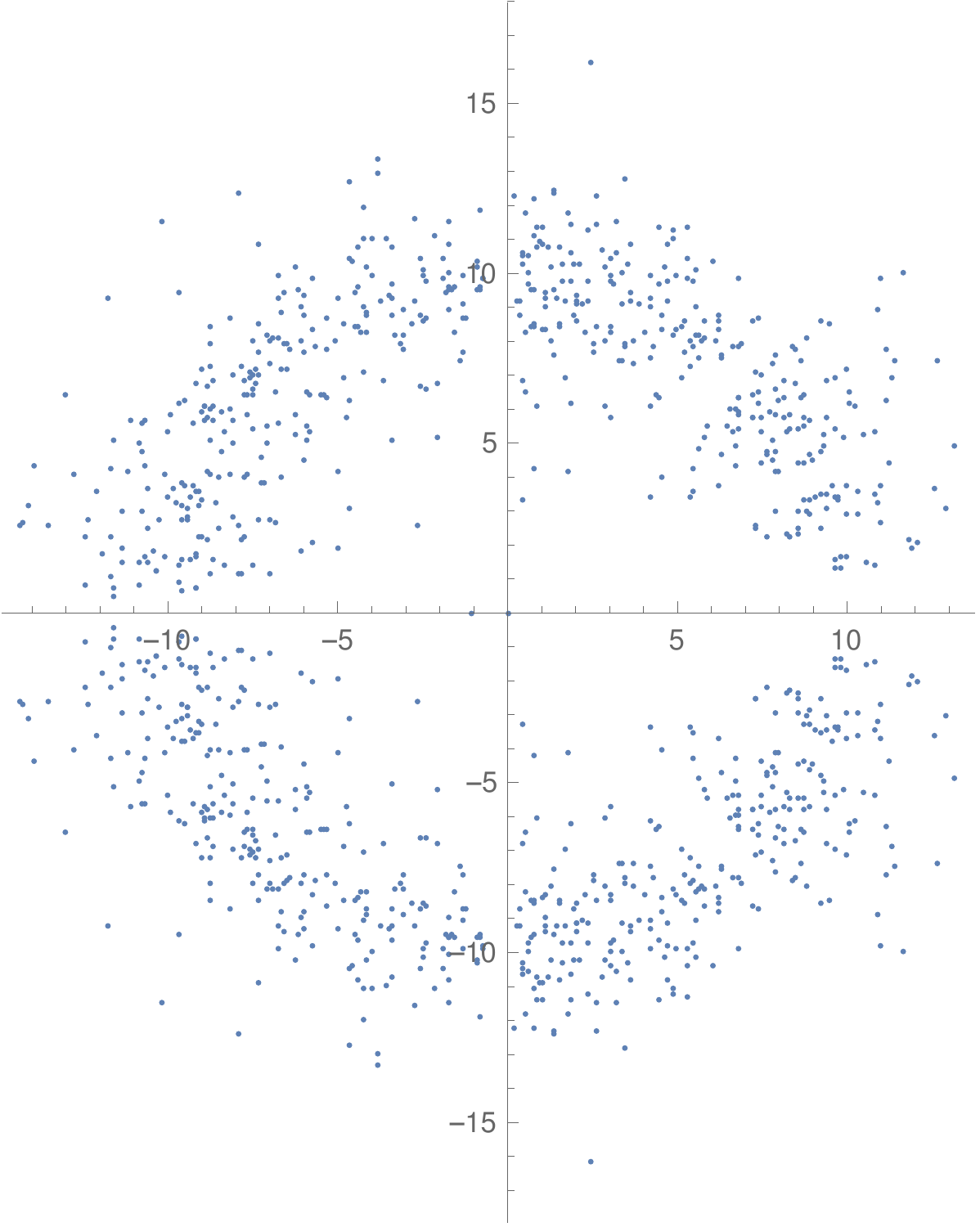}
 \caption{}\label{zeri}
\end{figure}

% \begin{tabular}{ c c c }
% -0.93753706$\pm$0.33482855 $i$, & -0.92536435$\pm$0.35488765 $i$, & -0.82600557$\pm$0.55115255 $i$,\\
% -0.68520628$\pm$0.67053411 $i$, & -0.55421602$\pm$0.82701155 $i$, & -0.36921064$\pm$0.92165325 $i$,\\
% 0.12031484$\pm$0.93460594 $i$, & 0.14635991$\pm$0.98490932 $i$, & 0.18459117$\pm$0.95835110 $i$,\\
% 0.39186275$\pm$0.89825762 i, & 0.74745993$\pm$0.65768729 $i$, & 0.77543377$\pm$0.61813413 $i$ 
% \end{tabular}

\vfill

Francesco Veneziano

Universit{\`a} di Genova

Via Dodecaneso 35, 16146 Genova, ITALY 

e-mail: veneziano@dima.unige.it

\FloatBarrier

% \bigskip
% 
% Umberto Zannier
% 
% Scuola Normale Superiore
% 
% Piazza dei Cavalieri, 7 -- 56126 Pisa -- ITALY \qquad 
% 
% e-mail: umberto.zannier@sns.it

\end{document}